\documentclass[a4paper]{amsart}

\usepackage{amssymb}
\usepackage{amsmath,amsthm}
\usepackage{enumerate,paralist}
\usepackage{amstext}
\usepackage{psfrag}
\usepackage{dsfont}
\usepackage[dvips]{epsfig}
\usepackage[english]{babel}
\usepackage{url}

\theoremstyle{plain}
\newtheorem{theorem}{Theorem}
\newtheorem{corollary}{Corollary}

\theoremstyle{definition}

\DeclareMathOperator{\sign}{sign} 

\DeclareMathOperator{\ctg}{ctg} \DeclareMathOperator{\tg}{tg}

\def\Sh{\mathop{\rm Sh}\nolimits}

\def\R{{\mathbb R}}

\def\Z{{\mathbb Z}}
\def\T{{\mathcal T}}

\def\F{{\mathfrak f}}

\def\myref#1{(\ref{#1})}

\def\text#1{\mbox{\rm #1\,}}

\def\theappend{\roman{append}}

\begin{document}
\frenchspacing

\title[Asymptotic control theory for a closed string]
{Asymptotic control theory for a closed string}

\author{Aleksey Fedorov}
\address{Russian Quantum Center\\
143025 Novaya st. 100, Skolkovo, Moscow, Russia\\
akf@rqc.ru}

\author{Alexander Ovseevich}
\address{Institute for Problems in Mechanics, Russian Academy of Sciences\\
119526, Vernadsky av., 101/1, Moscow, Russia\\
ovseev@ipmnet.ru}

\maketitle

\begin{abstract}
We develop an asymptotical control theory for one of the simplest distributed oscillating systems, namely, for a closed string under a bounded load applied to a single distinguished point. 
We find exact classes of string states that admit complete damping and an asymptotically exact value of the required time. 
By using approximate reachable sets instead of exact ones, we design a dry-friction like feedback control, which turns out to be asymptotically optimal. 
We prove the existence of motion under the control using a rather explicit solution of a nonlinear wave equation. 
Remarkably, the solution is determined via purely algebraic operations. 
The main result is a proof of asymptotic optimality of the control thus constructed.

\end{abstract}

\section{Introduction}

It is well-known that the limit capabilities of a control system can be described in terms of reachable sets, i.e. sets
of states reachable from a given state at a given time. In minimum-time problems it is convenient to pass to the backward
time. By passing to the backward time, we can consider sets reachable from the terminal manifold. Geometrically, the
minimum time control is organized as follows. Any state corresponds to a reachable set such that its boundary passes
through this state. The minimum-time control at the current state forces our system move in the direction normal to the
boundary of the corresponding reachable set.

Unfortunately, the reachable sets are difficult to study. However, sometime we possess an approximation of the reachable
sets. In particular, for linear systems there is a systematic asymptotic theory of the reachable sets as time goes to
infinity~\cite{Ovseevich2007}. By substituting the approximate reachable set for the exact reachable set we can design a
control. Alongside the simplicity this control might be asymptotically optimal. This type of control can be interpreted
as an action of a generalized dry friction. The behavior of any system of finite number of linear oscillators under the
dry-friction control has been studied in Refs~\cite{Ovseevich2013,Ovseevich2016}.

\subsection{Problem statement}

In this paper we apply this technique to control of a simple distributed system, the closed string under an impulsive
force applied at a fixed point in the string. The phase space $S$ of the system  consists of pairs
$\mathfrak{f}=(f_0,f_1)$ of distributions on a one-dimensional torus $\T$, and the motion is governed by the string
equation
\begin{equation}\label{string_eq}
    \frac{\partial^2 f}{\partial t^2}=\frac{\partial^2 f}{\partial x^2}+u\delta, \quad |u|\leq1.
\end{equation}
Here $x\in[0,2\pi]$ is the angle coordinate on the torus, $t$ is time, $f_0=f,$ $f_1=\frac{\partial f}{\partial t}$,
$\delta$ is the Dirac $\delta$-function. In other words,  $S$ is the space of initial data for Eq.~(\ref{string_eq}). It
is clear, that any solution of Eq.~(\ref{string_eq}) with zero initial data is even. This is why we assume that $S$
consists of pairs $\mathfrak{f}=(f_0,f_1)$ of even distributions. The equation describes oscillations of the closed
homogeneous string under a bounded load applied to a fixed point (zero). We note that a mechanical model of such system
can be not only a string, but also, e.g., a toroidal acoustic resonator.

Our goal is to design an easily implementable feedback control for damping the oscillations. This means that we do not
necessarily want to stop the motion of the string as a whole, so that that our target manifold ${\mathcal C}$ consists of
pairs of {constants}
\begin{equation}\label{eq:terminal}
    {\mathcal C}=\{(c_0,c_1)^*\in\R^2\subset S\}.
\end{equation}
Another useful point of view is to take the factor-space $\overline{S}=S/{\mathcal C}$ as the phase space of our system,
and try to reach zero in this space. This is reasonable, because the target space ${\mathcal C}$ is invariant under the
natural flow, associated with the string equation. In what follows, we deal with a class of problems of minimum-time
steering from the initial state to a terminal manifold ${\mathcal C}$ consisting of a pair of
constants~(\ref{eq:terminal}). More specifically, we study three problems:
\begin{enumerate}
    \item Complete stop at a given point: ${\mathcal C}=0$
    \item Stop moving: ${\mathcal C}=\R\times 0$
    \item Oscillation damping: ${\mathcal C}=\R^2$
\end{enumerate}

The main result of the present paper is the design of the generalized dry-friction control and the proof of its
asymptotic optimality in the stop moving problem. A summary of our results was presented in Ref.~\cite{Ovseevich2017}.

There are interesting related problems of stabilization of solutions to linear differential equations in partial
derivatives. The goal of stabilization is again damping, but it should be reached in  infinite time, and there are no a
performance index. We mention a few interesting and relatively recent papers on the subject
\cite{hassine,abdallah,rousseau}.
Moreover, problems of control for oscillating distributed systems are important for a wide spectrum of technological objects~\cite{Akulenko,Butkovskiy}.

\subsection{Structure}

The paper is organized as follows. In Section~\ref{sec:mechanics} we discuss a mechanical model of the considered system.
In Section~\ref{sec:support} we find an explicit form of the support functions of the reachable sets. We then study the
asymptotic behavior of the reachable sets as time goes to infinity using the language of shapes in
Section~\ref{sec:shape}. In Section~\ref{sec:control} on the basis of this study we design the generalized dry-friction
control for our system. In particular, we invoke the notion of the duality transform for a smooth description of the
generalized dry-friction. In Section~\ref{sec:restatement} we show that in the stop-moving problem the motion of the
system under the dry-friction control can be described via solution of a nonlinear wave equation of second order. We then
restate the nonlinear wave equation as a nonlinear first order equation. In Section~\ref{sec:optimality} we prove the
existence of the motion under the generalized dry-friction by means of a rather explicit solution of the nonlinear
equation. Remarkably, it is made via purely algebraic operations. In Section~\ref{sec:optimality} we demonstrate that the
suggested control is asymptotically optimal, and describe features of this control in Section~\ref{sec:features}.
Finally, we conclude in Section~\ref{sec:conclusion}. Appendices \ref{sec:AppI} and \ref{sec:AppIII} contain a number of
auxiliary results.

\section{String as a mechanical system.}\label{sec:mechanics}
The equations of motion of the free string \myref{string_eq} are that of the following Lagrangian system, where $q$ is an
even function such that  $\frac{\partial q}{\partial x}\in L_2(\T)$, and the Lagrangian
\begin{equation}\label{lagrange}
    L(q,\dot q)=\frac12\int_\T |\dot q|^2(x)\,dx-\frac12\int_\T \left|\frac{\partial q}{\partial x}\right|^2(x)\,dx-u q(0),
\end{equation}
so that $\frac12\int_\T |\dot q|^2\,dx$ is the kinetic energy, and $\frac12\int_\T \left|\frac{\partial f}{\partial
x}\right|^2\,dx$ is the potential energy of the system. The Lagrangian corresponds to the Hooke's law: the strain
(deformation) is proportional to the applied stress. In terms of the Lagrangian the stress at point $x$ is $\frac{\delta
L}{\delta q}(x)=\frac{\partial q}{\partial x}(x)$, and the strain at $x$ is $\frac{\partial q}{\partial x}(x)$, so that
the coefficient of proportionality is 1.

The string also allows for a Hamiltonian description. The phase space is then the set of pairs $(p,q)$, where $p$ is an
(even) function from $L_2(\T)$, $q$ is an (even) function from the space $N$ of functions such that $\frac{\partial
q}{\partial x}\in L_2(\T)$, and the Hamiltonian
\begin{equation}\label{lagrange_ham}
    H(p,q)=\frac12\int_\T |p|^2(x)\,dx+\frac12\int_\T \left|\frac{\partial q}{\partial x}\right|^2(x)\,dx+u q(0).
\end{equation}
The canonical symplectic structure $\omega=dp\wedge dq$ is given by
$\omega\left((X,X'),(Y,Y')\right)=\langle{X,Y'}\rangle-\langle{X',Y}\rangle$. Here $X,Y\in L_2(\T)$, $X',Y'\in N$, and
the angle brackets stand for the scalar product in $L_2(\T)$.

Finally, the Pontryagin Hamiltonian $H^{\rm pont}$ in ``coordinates'' $\mathfrak{f}=(f_0,f_1) $ takes the form
\begin{equation}\label{pont}
    H^{\rm pont}(\mathfrak{f},\mathfrak{x})=\langle f_1,\xi_0\rangle+\langle \Delta f_0,\xi_1\rangle+|\xi_1(0)|,
\end{equation}
with $\mathfrak{x}=(\xi_0,\xi_1)$ being the adjoint variables.

\section{The support function of reachable sets}\label{sec:support}

The first issue we deal with is that of controllability of the considered system. We approach it by computation of the
support function. The approach has much in common with that of Ref.~\cite{Lions1988}.

To make a comparison with the finite-dimensional case clear, we rewrite the governing equation in the form of first-order
system
\begin{equation}\label{string_eq2}
\begin{split}
    &\quad\quad\frac{\partial \mathfrak{f}}{\partial t}=A\mathfrak{f}+Bu, \quad |u|\leq1, \\
    A=&\left(
    \begin{array}{cc}
    0 & 1 \\
    \Delta & 0 \\
    \end{array}
    \right),
    \quad
    \Delta=\frac{\partial^2 }{\partial x^2},
    \quad
    B=\left(
    \begin{array}{c}
    0  \\
    \delta\\
    \end{array}
    \right).
\end{split}
\end{equation}
We commence with computing the support function $H=H_{D(T)}(\xi)$ of the reachable set $D(T),\,T\geq 0$ of
system~(\ref{string_eq2}) with zero initial condition. In other words, we have to find $H=\sup_u \langle
\mathfrak{f}(T),\xi\rangle$, where $\xi\in S^*$ is a dual vector, and the $\sup$ is taken over admissible controls.
Towards this end, we extend $\xi=(\xi_0,\xi_1)$ to the solution of the Cauchy problem
\begin{equation}\label{dual}
    \frac{\partial \xi}{\partial t}=-A^*\xi, \quad \xi(T)=\xi, \quad
    A^*=\left(
    \begin{array}{cc}
    0 & \Delta \\
    1 & 0 \\
    \end{array}
    \right)
\end{equation}
where $A^*$ is the adjoint operator to $A$. Note, in particular, that $\xi_1$ satisfies the wave equation. These
equations are exactly the equations for the adjoint variables of the Pontryagin maximum principle. We then have
\begin{equation}\
    \frac{d}{dt}\langle \mathfrak{f}(t),\xi(t)\rangle=\langle
    A\mathfrak{f}+Bu,\xi\rangle-\langle \mathfrak{f},A^*\xi\rangle=\langle Bu,\xi\rangle=u(t)\xi_1(0,t).
\end{equation}
A standard formal computation shows that
\begin{equation}\label{support}
    H=H_{D(T)}(\xi)=\sup_{|u|\leq1}\int_0^Tu(t)\xi_1(0,t)dt=\int_0^T|\xi_1(0,t)|dt.
\end{equation}

The reachable sets $D(T),\,T\geq 0$ are closed in the standard topology of distributions, and of course they are convex.
Therefore they are uniquely defined by their support functions.

Now we can characterize the vectors $\mathfrak{f}$ reachable from zero in time $T$ as follows:
\begin{equation}\label{reach}
    \mathfrak{f}\in{D(T)}\Leftrightarrow \langle \mathfrak{f},\xi\rangle\leq\int_0^T|\xi_1(0,t)|dt\quad \mbox{ for any }\xi\in S^*.
\end{equation}
Here the function $x,t\mapsto\xi_1(x,t)$ is defined via solution of the Cauchy problem~(\ref{dual}).

In particular, the space $\mathcal{D}$ generated by vectors $\mathfrak{f}\in\bigcup_{T\geq0}D(T)$ reachable from zero in
an arbitrary time $T\geq0$ is the dual space to the Frechet space of vectors $\xi$ with finite norms
\begin{equation}\label{norm}
    \Vert\xi\Vert=\Vert\xi\Vert_T=\int_0^T|\xi_1(0,t)|dt
\end{equation}
for any $T>0$. This space $\mathcal{D}$ coincides with the set of vectors reachable from zero in an arbitrary time
$T\geq0$ by means of a bounded (not necessarily by 1) control.

It is not difficult to compute $\xi_1(0,t)$ in terms of the Fourier coefficients of the functions $\psi=\xi_1$ and
$\phi=\xi_0$. Suppose that
\begin{equation}\label{fourier}
    \psi(x,t)=\sum_{-\infty}^\infty{\psi_n(t)e^{inx}}
\end{equation}
is the Fourier expansion of $\xi_1$. Since $\psi$ is an even and real distribution, the coefficients $\psi_n$ are real,
and $\psi_n=\psi_{-n}$, so that Eq.~(\ref{fourier}) is, in fact, the cosine-expansion:
\begin{equation}\label{cosine}
    \psi(x,t)=\sum_{n=0}^\infty{\psi_n(t)\cos{nx}}.
\end{equation}
The quantity we want to compute is
\begin{equation}
    \Vert\xi\Vert=\int_0^T|\xi_1(0,t)|dt=\int_0^T\left|\sum{\psi_n(t)}\right|dt.
\end{equation}
From Eq.~(\ref{dual}) we immediately conclude that for $n\neq0$
\begin{equation}\label{psi}
    \psi_n(t)=e^{int}a_n+e^{-int}b_n,
\end{equation}
where $a_n$, $b_n$ are constants. For $n=0$ we have $\psi_0(t)=a_0+b_0 t$. It is clear that for $n\neq0$
\begin{equation}\label{psi1}
    a_n=\frac{1}{2}\left(\psi_n+\frac{\phi_n}{in}\right), \quad b_n=\frac{1}{2}\left(\psi_n-\frac{\phi_n}{in}\right),
\end{equation}
where $\phi_n$ is the $n$th Fourier coefficient of $\phi$, and
\begin{equation}\label{psi2}
    a_0=\psi_0,\, b_0=\phi_0.
\end{equation}
The Fourier coefficients of $\phi$, like that of $\psi$, are real, and even with respect to $n$.

It follows from the above equations that
\begin{equation}\label{dalembert}
    \xi_1(0,t)=\sum_{n\neq0}\left(\psi_n\cos nt+\frac{\phi_n}{n}\sin nt\right)+\psi_0+\phi_0 t
\end{equation}
in agreement with the d'Alembert formula
\begin{equation}\label{dalembert2}
    \xi_1(x,t)=\frac12\left(\xi_1(x-t,0)+\xi_1(x+t,0)\right)-\frac12\int_{-t}^t \xi_0(y,0)dy
\end{equation}
for solution $\xi_1(x,t)$ of the wave equation.
\subsection{Natural norm in the dual space}\label{norm_section}
In view of \myref{dalembert} we conclude that
\begin{equation}\label{support2}
    \Vert\xi\Vert=\Vert\xi\Vert_T= \int_0^T\left|\sum_{n\neq0}\left(\psi_n\cos nt+\frac{\phi_n}{n}\sin nt\right)+\psi_0+\phi_0 t\right|dt.
\end{equation}
Suppose that $T$ is $\geq2\pi$. Then, the Banach norm $\Vert\xi\Vert$ is equivalent to the following  more familiar
Sobolev-type norm
\begin{equation}\label{prime}
    \Vert\xi\Vert'=\Vert\xi_1\Vert_{1}+\Vert\eta\Vert_{1}.
\end{equation}
Here $\Vert g\Vert_{1}=\int_{-T/2}^{T/2}\left|g\right|dt$ is the usual $L_1$-norm, and $\eta(t)=\int_0^t\xi_0(x)dx $.
Indeed, denote by $f$ the integrand
\begin{equation}
    f(t)=\sum_{n\neq0}\left(\psi_n\cos nt+\frac{\phi_n}{n}\sin nt\right)+\psi_0+\phi_0 t.
\end{equation}
The norm $\Vert\xi\Vert$ is equivalent to $\Vert
f\Vert_{1}=\int_{-T/2}^{T/2}\left|f\right|dt=\int_{-T/2}^{T/2}\left|f^++f^-\right|dt$, where
\begin{equation}
\begin{split}
    f^+(t)=\sum_{n\neq0}\psi_n\cos nt+\psi_0=\xi_1(t), \\
    f^-(t)=\sum_{n\neq0}\frac{\phi_n}{n}\sin nt+\phi_0t=\eta(t),
\end{split}
\end{equation}
are even and odd parts of the function $f$, respectively.

Indeed, put $g(t)=f(t)-\phi_0t$. Then, $\Vert\xi\Vert$  is equivalent to $\int_0^T |g|dt +|\phi_0|,$ while $\Vert
f\Vert_{1}$  is equivalent to $\int_{-T/2}^{T/2} |g|dt +|\phi_0|$. Since the function $g$ is $2\pi$-periodic, and
intervals of integration have length $T\geq2\pi$, both integrals $\int_0^T |g|dt,$ and $\int_{-T/2}^{T/2} |g|dt $ are
equivalent to $\int_{0}^{2\pi} |g|dt$.

The $L_1$-norms of  functions $f^\pm$ can be estimated via the $L_1$-norm of $f$:
\begin{equation}
    \Vert f^\pm\Vert_{1}\leq \Vert f\Vert_{1}.
\end{equation}
Therefore, $\Vert\xi\Vert'=\Vert f^+\Vert_{1}+\Vert f^-\Vert_{1}\leq2 \Vert f\Vert_{1}$. On the other hand, it is obvious
that
\begin{equation}
    \Vert f\Vert_{1}\leq\Vert f^+\Vert_{1}+\Vert f^-\Vert_{1}=\Vert\xi_1\Vert_{1}+\Vert\eta\Vert_{1}=\Vert\xi\Vert'.
\end{equation}

We conclude that the norms $\Vert\xi\Vert'$ and $\Vert\xi\Vert$ are equivalent indeed. Therefore, if $T\geq2\pi$ the dual
space to the Banach space with norm $\Vert\xi\Vert$ coincides with the space of pairs ${\mathfrak f}=(f_0,f_1)$, where
$\frac{\partial f_0}{\partial x}\in L_\infty$, and $f_1\in L_\infty$. Thus, it is possible to damp the string, where the
initial state ${\mathfrak f}=(f_0,f_1)$ possesses these properties, by a bounded load applied to a fixed point. Here by
damping we mean the complete stop, when not just oscillations, but also the displacement of the string as a whole is
forbidden.

\medskip
\noindent{\bf Remark.} Our arguments show that the equivalence class of the norm $\Vert\xi\Vert_T$ does not depend on $T$
provided that $T\geq2\pi$. Theorem~\ref{shape1} and Theorem~\ref{32} give a more quantitative statement of this
independence of $T$.

\subsection{Damping the oscillations}
In order to deal with damping oscillations only, it suffices to make an analog of previous computations in the
factor-space $S/{\mathcal C}$. The corresponding support functions are almost the same as those previously found. We just
have to assume that the zero-mode coefficients $\phi_0,\psi_0$ of the dual vectors $\xi$ are zero. Then, the formula for
support functions of the corresponding reachable sets $\overline{ D}(T)$ takes the following form:
\begin{equation}\label{support22}
    H_{\overline{ D}(T)}(\xi)=\int_0^T\left|\sum\left(\psi_n\cos nt+\frac{\phi_n}{n}\sin nt\right)\right|dt=\int_0^{T}\left|\xi_1(y)+\int_0^y\xi_0(x)dx\right|dy.
\end{equation}
Basically the same, but simpler arguments than that of previous subsection \ref{norm_section}, prove that states
${\mathfrak f}=(f_0,f_1)$, where $\frac{\partial f_0}{\partial x}\in L_\infty$, and $f_1\in L_\infty$ are exactly those
that can be damped.

\section{The shape of the reachable set $D(T)$}\label{sec:shape}

Following Ref.~\cite{Ovseevich2007}, one can easily find an asymptotic formula for the above support
function~(\ref{support2}). For $T>0$ define a linear  isomorphism $C(T):S\to S$ by
$C(T)\mathfrak{f}=\frac1T(f_0,f_1^T)^*$, where
\begin{equation}
    g^T(t)=\sum_{n\neq0} g_n\cos nt+\frac1T g_0, \mbox{ if } g(t)=\sum_{n\neq0} g_n\cos nt+ g_0.
\end{equation}
It is clear that
\begin{equation}\label{supportf}
    H_{C(T) D(T)}(\xi)=\frac{1}{T}\int_0^T\left|\sum_{n\neq0}\left(\psi_n\cos nt+\frac{\phi_n}{n}\sin nt\right)+\psi_0+\frac{\phi_0}{T}t\right|dt,
\end{equation}
where $\xi$ is the pair $(\xi_0,\xi_1)$, and $\xi_0(x)=\sum\phi_n\cos nx,\,\xi_1(x)=\sum\psi_n\cos nx$.

We then have the following precise statement.
\begin{theorem}\label{shape1}
    Consider problem (1) from the Introduction, corresponding to the terminal manifold ${\mathcal C}=0$.
    Then as $T\to\infty$ we have the following limit formula:
    \begin{equation}\label{asympt}
    \begin{split}
    \lim\limits_{T\to\infty}H_{C(T) D(T)}&(\xi)\\
    =&\frac{1}{2\pi}\int\limits_0^{2\pi}\int\limits_0^1\left|\sum_{n\neq0}\left(\psi_n\cos nt+\frac{\phi_n}{n}\sin nt\right)+\psi_0+\phi_0\tau\right|dtd\tau.
    \end{split}
    \end{equation}
\end{theorem}

\begin{proof}
The easy proof follows arguments from the basic results of Ref.~\cite{Ovseevich2007}, which we reproduce here for the
reader's convenience. It suffices to consider $T$ of the form $T=2\pi N,\,N\in\Z$. Then, we can rewrite
Eq.~(\ref{supportf}) in the following form:
\begin{equation}\label{supportf2}
    H_{C(T) D(T)}(\xi)=\frac{1}{2\pi N}\sum_{i=0}^{N-1}\int_0^{2\pi}\left|g(t)+{\phi_0}\left(\frac{i}{k}+\frac{t}{2\pi N}\right)\right|dt,
\end{equation}
where
\begin{equation}
    g(t)=\sum_{n\neq0}\left(\psi_n\cos nt+\frac{\phi_n}{n}\sin nt\right)+\psi_0
\end{equation}
is a $2\pi$-periodic function. We then note that the right-hand side of Eq.~(\ref{supportf2}) equals
\begin{equation}\label{supportf3}
    \frac{1}{2\pi N}\sum_{i=0}^{N-1}\int_0^{2\pi}\left|g(t)+{\phi_0}\frac{i}{k}\right|dt+O\left(\frac1N\right),
\end{equation}
which is the Riemann sum for the integral over $[0,1]$ of the continuous function
\begin{equation}
    \tau\mapsto \frac{1}{2\pi}\int_0^{2\pi}\left|g(t)+{\phi_0}\tau\right|dt.
\end{equation}
Since the Riemann sums converge to the integral
\begin{equation}
    \frac{1}{2\pi}\int_0^1\int_0^{2\pi} \left|g(t)+{\phi_0}\tau\right|dtd\tau \mbox{ as } N\to\infty,
\end{equation}
the statement is proved.
\end{proof}

We remind that the shape $\Sh\Omega$ of a set $\Omega\subset S$ is the orbit of the group of linear (topological)
isomorphisms of the space $S$ acting on $\Omega$. In terms of shapes one can say that the limit shape
\begin{equation}
    \Sh_{\infty}=\lim\limits_{T\to\infty}\Sh D(T)
\end{equation}
is related to the convex body $\Omega$ corresponding to the support function
\begin{equation}\label{asympt2}
    H_{\Omega}(\xi)=\frac{1}{2\pi}\int_0^{2\pi}\int_0^1\left|\sum_{n\neq0}\left(\psi_n\cos nt+\frac{\phi_n}{n}\sin nt\right)+\psi_0+\phi_0\tau\right|dtd\tau.
\end{equation}
This means that the shape of the convex set with the support function (\ref{asympt2}) coincides with $\Sh_{\infty}$.
Similarly, in the reduced space $\overline{S}$ we have the following precise result.
\begin{theorem}\label{32}
Consider problems (2)-(3) from the Introduction, corresponding to the terminal manifolds ${\mathcal C}=\R\times 0$, or
${\mathcal C}=\R^2$. Then, the following limit formula holds:
\begin{equation}\label{asympt3}
\begin{split}
    &\lim\limits_{T\to\infty}\frac1TH_{\overline{D}(T)}(\xi)=\\
    &\frac{1}{2\pi}\int_0^{2\pi}\left|\sum\left(\psi_n\cos nt+\frac{\phi_n}{n}\sin nt\right)\right|dt=\frac{1}{2\pi}\int_0^{2\pi}\left|\zeta(t)\right|dt.
\end{split}
\end{equation}
where $\zeta(t)=\xi_1(t)+\int_0^t\xi_0(x)dx$.
\end{theorem}
Note that the operator of multiplication by $(T)^{-1}$ in the factor-space $\overline S$ is induced by the operator
$C(T)$, and Eq.~(\ref{asympt3}) describes the limit shape $\lim\limits_{T\to\infty}\Sh \overline{D}(T)$. Denote by
$\Omega$ the convex body such that its support function is given by the right-hand side of \myref{asympt3}:
\begin{eqnarray}\label{asympt4}
    H_{\Omega}(\xi)=\frac{1}{2\pi}\int_0^{2\pi}\left|\sum\left(\psi_n\cos nt+\frac{\phi_n}{n}\sin nt\right)\right|dt=\frac{1}{2\pi}\int_0^{2\pi}\left|\zeta(t)\right|dt.
\end{eqnarray}
According to Theorem~\ref{32} the set $T\Omega$  is an approximation of $D(T)$ if $T$ is large.

\section{Dry-friction control}\label{sec:control}

Our control design is based on the following idea: The optimal control at state ${\mathfrak f}$ implements the steepest
descent in the direction normal to boundaries of the reachable sets $D(T)$. Our control implements the steepest descent
in the direction normal to boundaries of the approximate reachable sets $T\Omega$, where $\Omega$ is defined via
Eq.~(\ref{asympt4}). This means that in notations of Eq.~(\ref{asympt3}) we have
\begin{equation}\label{Legendre20}
    u({\mathfrak f})=-\sign\langle B,\xi\rangle=-\sign\xi_1(0)=-\sign\zeta(0),
\end{equation}
where the momentum $\xi$ is to be found via the equation
\begin{equation}\label{asympt5}
    T^{-1}{\mathfrak f}=\frac{\partial  H_{\Omega}}{\partial\xi}(\xi),
\end{equation}
or, equivalently, ${\mathfrak f}=(f_0,f_1)$, where
\begin{equation}\label{asympt6}
    T^{-1}f_0(x)=-\int_0^{x}(\sign \zeta(y))^-dy,\quad T^{-1}f_1(x)=(\sign \zeta(x))^+,
\end{equation}
where the notation $f^\pm$ stands for even/odd part of the function $f$:
\begin{equation}\label{pm}
    f^\pm(x)=\frac12(f(x)\pm f(-x)).
\end{equation}

These identities are to be understood as inclusions, because the $\sign$-map is  multivalued. Namely, their precise
meaning is
\begin{equation}\label{asympt7}
    T^{-1}f_0(x)=-\int_0^{x}\phi(y)^-dy,\quad T^{-1}f_1(x)=\psi(x)^+,
\end{equation}
where $\phi(y)\in\sign \zeta(y),$ and $\psi(x)\in\sign \zeta(x)$.

\subsection{Duality transform}
We invoke a general duality transformation related to Eq.~(\ref{asympt5}). Toward this end, we denote the function
$H_\Omega$ just by $H=H(\xi)$, and the factor $T$ by $\rho({\mathfrak f})$. Then the relation between $H$ and $\rho$ is
similar to the Legendre transformation:
\begin{equation}\label{Legendre}
    \langle {\mathfrak f},\xi\rangle=\rho({\mathfrak f}) H(\xi),
    \quad
    \rho({\mathfrak f})=\max_{H(\xi)\leq1}\langle{\mathfrak f},\xi\rangle,
    \quad
    H(\xi)=\max_{\rho({\mathfrak f})\leq1}\langle {\mathfrak f},\xi\rangle,
\end{equation}
where the correspondence ${\mathfrak f}\rightleftarrows \xi$ has the following form:
\begin{eqnarray}
    {\mathfrak f}=\rho({\mathfrak f})\frac{\partial H}{\partial \xi}(\xi), \label{Legendre1}
    \\
    \xi=H(\xi)\frac{\partial \rho}{\partial{\mathfrak f}}({\mathfrak f}). \label{Legendre2}
\end{eqnarray}
Here $\xi$ and ${\mathfrak f}$ are the points where the maximums in (\ref{Legendre}) are attained. These relations make
sense provided that $H$ and $\rho$ are norms, i.e. homogeneous of degree 1 convex  functions such that the sublevel sets
$\{H(\xi)\leq1\}$ and $\{\rho({\mathfrak f})\leq1\}$ are convex bodies. These sublevels are mutually polar to each other.
In other words, if $\Omega=\{\rho({\mathfrak f})\leq1\}$, and ${\Omega}^{\circ}=\{H(\xi)\leq1\}$, then
$\Omega=\{{\mathfrak f}:\langle {\mathfrak f},\xi\rangle\leq1,\,\xi\in{\Omega}^\circ\}$ and vice versa. In the language
of Banach spaces, the normed spaces $(\mathbb{V},\rho)$ and $(\mathbb{V}^*,H)$ are dual to each other. The derivatives in
Eq.~(\ref{Legendre2}) should be understood as subgradients. If the functions $H$ and $\rho$ are differentiable then
Eq.~(\ref{Legendre2}) has the classical meaning. If one of the functions $H$ and $\rho$ is differentiable and strictly
convex, then, the other one is also so.

In the cases at hand we need to calculate the dual function $\rho$ for the function $H=H_\Omega$ from
Eq.~(\ref{asympt4}). We then arrive at the following result.

\begin{theorem}\label{rho00}
Consider problems (2) -- (3) from the Introduction, corresponding to the terminal manifolds ${\mathcal C}=\R\times 0$, or
${\mathcal C}=\R^2$.
\begin{enumerate}
    \item If ${\mathcal C}=\R^2$, then $\rho({\mathfrak f})={2\pi}\left|\frac{\partial f_0}{\partial x}+f_1\right|_\infty$,
    where the norm $\displaystyle\left|\phi\right|_\infty$ of a function $\phi$ on the torus $\T=\R/2\pi\Z$ is $\displaystyle\inf_{c}\sup_{x\in\T}|\phi(x)+c|$,
    where $c\in\R$ is an arbitrary constant.
    \item If ${\mathcal C}=\R\times0$, then $\rho({\mathfrak f})={2\pi}\left|\frac{\partial f_0}{\partial x}+f_1\right|_\infty$,
    where the norm $\displaystyle\left|\phi\right|_\infty$ is the $\sup$-norm of the function $\phi$ on the torus $\T=\R/2\pi\Z$.
\end{enumerate}
\end{theorem}

\begin{proof}
We consider only the problem (2) of Theorem~\ref{rho00}, because the problem (3) is quite similar. We note that in case
(1)
\begin{equation}\label{infty_norm}
    \left|\phi\right|_\infty=\frac12\left(\sup \phi-\inf \phi\right).
\end{equation}
We note that $\frac{\partial f_0}{\partial x}$ is an odd function, while $f_1$ is even. This implies that the norm
$\rho({\mathfrak f})$ is equivalent to (although does not coincide with) $\max\left(\left|\frac{\partial f_0}{\partial
x}\right|_\infty,\left|f_1\right|_\infty\right)$.

To prove the theorem we define $\rho_0({\mathfrak f})$ by the formula
\begin{equation}
    \rho_0({\mathfrak f})=\left|\frac{\partial f_0}{\partial x}+f_1\right|_\infty,
\end{equation}
and check that $H(\xi)=\frac{1}{2\pi}\max\limits_{\rho_0({\mathfrak f})\leq1}\langle {\mathfrak f},\xi\rangle$. Toward
this end, put
\begin{equation}
    \psi_0(t)=\int_0^t\xi_0(x)dx,\,\psi_1(t)=\xi_1(t),\,\psi=\psi_0+\psi_1,
\end{equation}
and
\begin{equation}
    \phi_0=\frac{\partial f_0}{\partial x},\,\phi_1=f_1,\,\phi=\phi_0+\phi_1.
\end{equation}

Denote by $\int_\T f$, where $\T=\R/2\pi\Z$, the normalized integral $\frac{1}{2\pi}\int_0^{2\pi} f(t)dt$. Remind that
$\langle {\mathfrak f},\xi\rangle$ stands for $\int_0^{2\pi} \langle {\mathfrak f}(t),\xi(t)\rangle dt$.

Then,
\begin{equation}\label{identity}
    \int_\T \phi\psi=\int_\T \phi_0\psi_0+\int_\T \phi_1\psi_1=\frac{1}{2\pi}\langle {\mathfrak f},\xi\rangle,
\end{equation}
because the integrals $\int_\T \phi_0\psi_1dt,\,\int_\T \phi_1\psi_0 dt$ vanish, being integrals of odd functions over
$\T=\R/2\pi\Z$. It is clear from Eq.~(\ref{identity}), that the maximum of $\int_\T \phi\psi dt$, taken over $\phi$ such
that $|\phi|_\infty\leq1$, coincides with the maximum of $\langle {\mathfrak f},\xi\rangle$, taken over ${\mathfrak f}$
such that $\rho_0({\mathfrak f})\leq1$. However, it is trivial that the maximum of $\int_\T \phi\psi =\int_\T |\psi| $.
The latter value, according to Eq.~(\ref{asympt4}), equals $H(\xi)$.
\end{proof}

One can regard our computation of the norm $\rho$ as an a priori estimate for solutions of the wave equation.
\begin{theorem}\label{RhoEstimate}
    Suppose ${\mathfrak f}=(f_0,f_1),$ is a solution of the Cauchy problem
    \begin{equation}\label{string_eq22}
    \frac{\partial\mathfrak{f}}{\partial t}=A\mathfrak{f}+Bu, \quad |u|\leq1,\quad \mathfrak{f}(0)=0,
    \end{equation}
    where $B=(0,\delta)$, $u=u(t)$.
    If $T\geq2\pi$ we then have
    \begin{equation}
        \rho(\mathfrak{f}(T))={2\pi}\left|\frac{\partial f_0}{\partial x}+f_1\right|_\infty\leq{T}.
    \end{equation}
\end{theorem}
In particular, the following a priori bound holds true:
\begin{corollary}\label{RhoBound}
    Suppose ${\mathfrak f}=(f_0,f_1),$ is a solution of
    $\frac{\partial \mathfrak{f}}{\partial t}=A\mathfrak{f}+Bu$, $|u|\leq1$,
    while $\widetilde {\mathfrak f}$ is control-free: $\frac{\partial \mathfrak{f}}{\partial t}=A\mathfrak{f}$,
    and $\widetilde {\mathfrak f}(0)={\mathfrak f}(0)$.
    Then provided that $T\geq2\pi$ we have
    \begin{equation}
        \left|(f_1-\widetilde{f}_1)(T)\right|_\infty\leq \frac1{2\pi}T.
    \end{equation}
\end{corollary}

\section{Computation of the basic control}\label{sec:computation}

Consider Problem (2) from the Introduction, corresponding to the terminal manifolds ${\mathcal C}=\R\times0$. In order to
find the control we need to solve Eqs.~(\ref{asympt6}) as explicitly as possible. In other words, we have to find
function $\zeta$ such that
\begin{equation}\label{asympt8}
    T^{-1}\frac{\partial f_0}{\partial x}(x)\in-\sign \zeta(x),\quad T^{-1}f_1(x)\in\sign \zeta(x).
\end{equation}
Our discussion of duality, in particular Eq.~\myref{Legendre2} shows that the solution is given by
\begin{equation}
    T=\rho({\mathfrak f})={2\pi}\left|\frac{\partial f_0}{\partial x}+f_1\right|_\infty,\quad\zeta=\frac{\partial\rho}{\partial {\mathfrak f_1}}({\mathfrak f}).
\end{equation}

The final expression for the control has the following form:
\begin{eqnarray}\label{control}
    u\left(\F\right)=-\sign\zeta(0)=-\sign f_1(0),
\end{eqnarray}
where we take into account Eq.~(\ref{asympt6}) and the vanishing at $0$ of the odd part of the function
$x\mapsto\sign\zeta(x)$. Thus, we obtain indeed a generalization of the dry friction, for it acts with maximal possible
amplitude against the velocity, because $f_1(0)$ is exactly the velocity of the point, where the load is applied. Control
(\ref{control}) leads to the nonlinear wave equation of the following form:
\begin{equation}\label{string_eq220}
    \frac{\partial^2 f}{\partial t^2}=\frac{\partial^2 f}{\partial x^2}-\sign\left(\frac{\partial f}{\partial t}(0)\right)\delta
\end{equation}
governing the damping process. There are no a standard existence and uniqueness theorem for the Cauchy problem for
Eq.~(\ref{string_eq220}). The problem is to extend to infinite-dimensional case the classical results of
Filippov~\cite{Filippov1988} concerning existence, and the uniqueness theorem of Bogaevsky~\cite{Bogaevskii2006}.

\section{Restatement of the model}\label{sec:restatement}

Previous considerations stress the importance of the function
\begin{equation}\label{g}
    g=\frac{\partial f_0}{\partial x}+f_1.
\end{equation}
Knowledge of this function is almost equivalent to the knowledge of both functions $f_0$ and $f_1$. Indeed, the function
${\partial f_0}/{\partial x}$ is odd and $f_1$ is even. Therefore, knowledge of these functions is equivalent to the
knowledge of the function $g$. On the other hand the knowledge of ${\partial f_0}/{\partial x}$ gives a complete
information on $f_0$ up to an additive constant. This constant is irrelevant if the goal of our damping process is to
stop oscillation, or to stop motion of the string at an unspecified point. The law of the controlled motion given by
Eqs.~(\ref{string_eq}), (\ref{control}), and (\ref{string_eq220}) can be restated as follows:
\begin{equation}\label{model}
    \left(\frac{\partial }{\partial t}-\frac{\partial}{\partial x}\right)g(x,t)=\delta(x)u(t),\quad |u|\leq1.
\end{equation}
This form of the governing law has its merits. In particular, it can be made rather explicit: One can rewrite
Eq.~(\ref{model}) as follows:
\begin{equation}\label{model2}
    \frac{d}{d t}\,g(x-t,t)=\delta(x-t)u(t),
\end{equation}
which means that
\begin{equation}\label{model3}
    g(x-t,t)=g(x,0)+\int_0^t\delta(x-s)u(s)ds=g(x,0)+\sum_{I} u(x+2k\pi),
\end{equation}
where the summation is over the set $I={I_t}$ of $k\in\Z$ such that $x+2k\pi\in[0,t]$. By the change of variables $z=x-t$
we arrive at:
\begin{equation}\label{model4}
    g(z,t)=g(z+t,0)+\sum_{J} u(z+t+2k\pi),
\end{equation}
where the summation is over the set $J={J_t}$ of $k\in\Z$ such that $z+2k\pi\in[-t,0]$.

Eq.~(\ref{model4}) should be understood as follows: Here $g$ is a bounded measurable function of $x,t$ and $u$ is a
bounded measurable function of $t$, the curve $t\mapsto g(\cdot,t)$ is continuous as a map from real line to
distributions depending on the space variable $x$. Eq.~(\ref{model4}) does not hold pointwise, but expresses an equality
in the space of curves of distributions wrt $x$.

\section{Existence of the motion under dry-friction control}\label{sec:existence}

We have to obtain an existence theorem for  initial value problem for the nonlinear wave equation~(\ref{string_eq220}).
By using transformation \myref{g} the task reduces to solution of the functional equation
\begin{equation}\label{newmodel5}
    g(z,t)=g(z+t,0)-\sum_{J} \sign g(0,z+t+2k\pi),
\end{equation}
which in turn can be reduced to the search for the function $g(0,t),\,t\geq0$, because this defines the control low
$u(t)=-\sign g(0,t)$.

This is quite nontrivial, because the the function $\phi(t)=g(0,t)$, we are looking for, should satisfy a functional
equation. The first step in establishing the desired functional equation is to make Eq.~(\ref{newmodel5}) hold pointwise.
It is explained in the previous section that this equation expresses an equality in the space of curves of distributions
of $x$. In order to make Eq.~(\ref{newmodel5}) hold pointwise we consider the one-sided averaging operators
\begin{equation}\label{average}
    {\rm Av}^{\pm\epsilon}: f(z,t)\mapsto \frac1\epsilon\int_0^{\pm\epsilon}f(z+x,t)dx,\quad {\rm Av}^{\pm}: f\mapsto\lim_{\epsilon\to0}{\rm Av}^{\pm\epsilon}(f),
\end{equation}
and more standard two-sided operator
\begin{equation}\label{average2}
    {\rm Av}^{\epsilon}: f(z,t)\mapsto \frac1{2\epsilon}\int_{-\epsilon}^{\epsilon}f(z+x,t)dx,\quad {\rm Av}: f\mapsto\lim_{\epsilon\to0}{\rm Av}^{\epsilon}(f).
\end{equation}
We note that, according to the Lebesgue differentiation theorem, the limit averaging operators ${\rm Av}^{\pm}$ and ${\rm
Av}$ are identities when applied to any $L_1$-function. The reason for application of these operators is that, if
operators ${\rm Av}^{\pm\epsilon}$ are applied to the right-hand and the left-hand sides of Eq.~(\ref{newmodel5}), the
obtained equation holds {\em pointwise}. In particular,
\begin{equation}\label{newmodel5average}
\begin{array}{l}
    {\rm Av}^{\epsilon}(g)(0,t)={\rm Av}^{\epsilon}(g)(t,0)-\\[1em]\sum\frac1{2\epsilon}\int_{-\epsilon}^{\epsilon}u(z+t+2k\pi)1_{[-t,0]}(z+t+2k\pi)dz= \\[1em]
    {\rm Av}^{\epsilon}(g)(t,0)-\frac12
    {\rm Av}^{-\epsilon}(u)(t)-\sum_{k\neq0}
    {\rm Av}^{\epsilon}(u)(t+2k\pi),
\end{array}
\end{equation}
where $u(t)=-\sign g(0,t)$, summation is over the set  of $k\in\Z$ such that $2k\pi\in[-t,0]$, and $\epsilon<t$.

To state the desired functional equation consider the function $\phi(t)={\rm Av}(g)(0,t)$. It follows from
Eq.~(\ref{newmodel5average}) by passing to the limit $\epsilon\to0$ that this function does exist in
$L_\infty(0,\infty)$. It also follows from Eq.~(\ref{newmodel5average}) that
\begin{equation}\label{model501}
    \phi(t)=G(t)- \frac12\sign \phi(t)-\sum_{k\neq0,2k\pi\in[-t,0]}\sign \phi (t+2k\pi),
\end{equation}
where $G(t)=g(t,0)$ is the given initial  $2\pi$-periodic function. Solution of this  equations  gives at the same time a
rigorously defined solution to the nonlinear wave equation~(\ref{string_eq220}).

Thus, we have to solve the equation
\begin{equation}\label{model501}
    \phi(t)+  \frac12\sign \phi(t)+\sum_{k\neq0,2k\pi\in[-t,0]}\sign \phi (t+2k\pi)=G(t),
\end{equation}
where $G$ is a given function, and $\phi$ is unknown. Note that the function $\phi$ need not be periodic. It should be
defined for nonnegative $t$. Note also that if $t<2\pi$, the latter equation reduces to a very simple one:
\begin{equation}\label{model502}
    \phi(t)+ \frac12\sign \phi(t)=G(t),
\end{equation}
which, obviously, has a unique solution, since the map $x\mapsto x+\sign x$ is a strictly monotone increasing
(multivalued) function. More explicitly, the solution $\phi(t)=G(t)-\frac12$ if $G(t)>\frac12$, and
$\phi(t)=G(t)+\frac12$ if $G(t)<-\frac12$. Otherwise, $\phi(t)=0$. Note that $|\phi(t)|\leq |G(t)|$  in the considered
interval $[0,2\pi)$ of values of the argument $t$.

It is better rewrite the above equation \myref{model502} in the form
\begin{equation}\label{model502+}
    \phi(t)+ \frac12 v(t)=G(t),\quad  v(t)=\sign \phi(t),
\end{equation}
where $\sign$-function is regarded as multivalued: $\sign(0)=[-1,1]$. Then, the a priori multivalued $\sign \phi(t)$ is
defined by (\ref{model502+}) uniquely. If $t<2\pi$ we obtain from Eq.~(\ref{model501}) and periodicity $G(t+2\pi)=G(t)$
that
\begin{equation}\label{phi3}
    \phi(t+2\pi)+ \frac12\sign \phi(t+2\pi)=G(t)-\sign \phi(t),
\end{equation}
which allows to extend by the preceding arguments the function $\phi(t)$ from $t\in[0,2\pi)$ to any positive value of
$t$. By the already used arguments we obtain that $|\phi(t)|\leq |G(t)|$ for all  $t\geq0$. We then have the following
precise statement.
\begin{theorem}\label{motion}
    The Cauchy problem for the nonlinear wave equation
    \begin{equation}\label{model_theorem} \left(\frac{\partial }{\partial t}-\frac{\partial}{\partial x}\right)g(x,t)=-\delta(x)\sign g(0,t),
\end{equation}
where $g(x,0)$ is a given bounded (Borel-measurable) function possesses a unique bounded solution for $t\geq0$. The
functions $\phi(t)=g(0,t)$ and $u(t)=-\sign g(0,t)$ form  a unique solution of functional equation~(\ref{model501}).
\end{theorem}
We call the flow $g=g(\cdot,0)\mapsto\Phi_t(g)=g(\cdot,t)$, where $t\geq0$, in the space of measurable bounded functions
the dry-friction flow.

\section{Asymptotic optimality: proof}\label{sec:optimality}

\subsection{Asymptotic optimality of control: polar-like coordinate system}
Here we present at an intuitive level reasons for asymptotic optimality of the control law~(\ref{sec:control}). The
rigorous treatment of asymptotic optimality is performed in the below. We define a polar-like coordinate system, well
suited for repre\-sen\-ta\-tion of the motion under the control $u$. Every state $0\neq \mathfrak{f}$ of the string can
be represented uniquely as
\begin{equation}\label{rhophi}
    \mathfrak{f}=\rho\phi,\mbox{ where }\rho=\rho(x)\mbox{ is a positive factor, and }\phi\in\partial\Omega.
\end{equation}
The pair $\rho,\phi$ is the coordinate representation for $x$, and $\rho(\phi)=1$ is the equation of the ``sphere''
$\omega=\partial\Omega$. It is important that the set $\omega$ is invariant under free (uncontrolled) motion of our
system~(\ref{string_eq2}). This follows from the similar invariance of the support function ${H}_{\Omega}(p)$ under
evolution governed by $\dot{p}=-{A^*}p$. The latter invariance is clear, because the support function is an ergodic mean
of the function $|\xi_1(0,t)|$ under the free motion. This implies invariance of  the  dual function
$\rho=\rho(\mathfrak{f}),$ so that $\left\langle{\partial{\rho}}/{\partial\mathfrak{f}},A\mathfrak{f}\right\rangle=0$.
Therefore, under the control $u$ from \myref{control} the total (Lie) derivative of $\rho$ takes the following form:
\begin{equation}\label{T}
    \dot \rho=\left\langle{\frac{\partial {\rho}}{\partial \mathfrak{f}},A\mathfrak{f}+Bu}\right\rangle=
    \left\langle{\frac{\partial {\rho}}{\partial \mathfrak{f}},Bu}\right\rangle=-\left|\left\langle{\frac{\partial{\rho}}{\partial
    \mathfrak{f}},B}\right\rangle\right|,
\end{equation}
where the last identity holds because ${\partial {\rho}}/{\partial \mathfrak{f}}$ is the outer normal to the set
$\rho\Omega$. In particular, the ``radius'' $\rho$ is a monotone non-increasing function of time. Moreover, the RHS of
Eq.~(\ref{T}) necessarily equals -1 if $f_1(0)\neq0$. For any other admissible control, we have
\begin{equation}\label{T3}
    \dot\rho\geq-\left|\left\langle{\frac{\partial {\rho}}{\partial\mathfrak{f}},B}\right\rangle\right|.
\end{equation}
The evolution of $\phi$ by virtue of system (\ref{string_eq2}) is described by
\begin{equation}\label{xi}
    \dot \phi=A\phi+\frac{1}{\rho}(Bu-\phi \dot \rho)
    =A\phi+\frac{1}{\rho}\left(Bu+\phi \left|\left\langle{\frac{\partial{\rho}}{\partial\mathfrak{f}},B}\right\rangle\right|\right).
\end{equation}
We note that the right-hand side
$-\left|\left\langle{\frac{\partial{\rho}}{\partial\mathfrak{f}}(\mathfrak{f}),B}\right\rangle\right|$ of Eq.~(\ref{T})
equals $-\left|\left\langle{\frac{\partial{\rho}}{\partial\mathfrak{f}}( \phi),B}\right\rangle\right|$. Thus, the
evolution of the RHS of Eq.~(\ref{T}) is determined by the evolution of $\phi$ by virtue of Eq.~(\ref{xi}). It is clear
that if $\rho$ is large, then the second term in the RHS of (\ref{xi}) is $O(1/\rho)$ and affects the motion of $\phi$
over the ``sphere'' $\omega$ only slightly. Our next task is to compute approximately the  ``ergodic mean''
\begin{equation}
    E_T=\frac{1}{T}\int_0^T \left|\left\langle{\frac{\partial{\rho}}{\partial\mathfrak{f}},B}\right\rangle\right|dt
\end{equation}
of the RHS of Eq.~(\ref{T}) provided that $\rho$ is large. Here $B$ is a constant vector, while, according to the
preceding arguments, the vector function
\begin{equation}
    \frac{\partial {\rho}}{\partial \mathfrak{f}}(t):=\frac{\partial {\rho}}{\partial \mathfrak{f}}(\mathfrak{f}(t))
\end{equation}
behaves approximately as $e^{A^*t}\frac{\partial {\rho}}{\partial \mathfrak{f}}(0)$. Therefore, the ergodic mean $E_T$ is
well approximated by
\begin{equation}
    E_T=\frac{1}{T}\int_0^T \left|\left\langle{e^{A^*t}\xi,B}\right\rangle\right|dt,
\end{equation}
where $\xi=\frac{\partial {\rho}}{\partial \mathfrak{f}}(0)$. We know from Theorem \ref{32} that as $T\to\infty$ the
ergodic mean $E_T$ tends to $H(\xi)=H(\frac{\partial {\rho}}{\partial \mathfrak{f}})$. However, according to one of the
basic ``duality relation'' (\ref{Legendre2}), we know that $H(\frac{\partial {\rho}}{\partial \mathfrak{f}})=1$.

Therefore, we conclude, by using abbreviation $\rho(t)=\rho(\frac{\partial {\rho}}{\partial
\mathfrak{f}}(\mathfrak{f}(t)))$, that
\begin{equation}\label{asymp_1}
    {(\rho(0)-\rho(T))}/{T}=1+o(1), \mbox{ as }T\to\infty ,
\end{equation}
provided that we use the dry-friction control~(\ref{control}).

Under any other admissible control, according to Theorem \ref{RhoEstimate},
\begin{equation}\label{asymp_2}
    {(\rho(0)-\rho(T))}/{T}\leq 1+o(1).
\end{equation}
These latter relations (\ref{asymp_1}) and (\ref{asymp_2}) express the asymptotic optimality we sought for.

\subsection{Formal proof}

Here we prove the asymptotic optimality of control~(\ref{control}) via the use of the function $g(x,t)$ from
Eq.~(\ref{g}). The law of motion~(\ref{model4}) is
\begin{equation}\label{model5}
    g(z,t)=g(z+t,0)-\sum_{J} \sign g(0,z+t+2k\pi),
\end{equation}
where the set $J={J_t}$ consists of $k\in\Z$ such that $z+2k\pi\in[-t,0]$. The functional $\rho$ has the form
$\rho(g)=2\pi\sup_{x\in \R/2\pi\Z}|g(x,t)|$. The control $\sign g(0,z+t)$ is not affected by the scaling transformation
\begin{equation}
    g\mapsto \Phi=g/\rho.
\end{equation}
However, if $\rho$ is large, then our previous considerations reveal that the function $\Phi=g/\rho$ moves in an almost
uncontrollable mode. The latter means that approximately
\begin{equation}
    \Phi(x,t)\thickapprox\Phi(x+t,0),
\end{equation}
so that we come to the approximate equality $$\sign g(0,z)\thickapprox \sign g(z,0).$$

More precisely, suppose that in the time-interval $[0,T]$ we have $\rho(g_t)\geq 2\pi M,$ where $M$ is a (large)
constant. In view of Eq.~(\ref{model501}) we have
\begin{equation}\label{model50}
    g(0,t)=g(t,0)-\frac12 \sign g(0,t) -\sum_{k\neq0,\,2k\pi\in[-t,0]}\sign g(0,t+2k\pi),
\end{equation}
and, therefore,
\begin{equation}\label{model51}
    |g(0,t)-g(t,0)|\leq \frac{t}{2\pi}.
\end{equation}
Since $\rho(g)\geq M$  there exist points $x\in\R/2\pi\Z$, where either $g(x,0)\geq M-1$ or $g(x,0)\leq -(M-1)$. Assume
for definiteness that $g(x,0)\geq M-1$. Then in view of Eq.~(\ref{model51}), $\sign g(0,t+2k\pi)=+1$ for $t\in[0,T]$
provided that $\frac{T}{2\pi}\leq M-1$. For instance, this is the case if $M$ is large and $T=O(\sqrt{M})$.

In view of \myref{model5} this means that
\begin{equation}\label{model6}
    g(z-t,t)= g(z,0)-\frac{t}{2\pi} \sign g(z,0)+O(1),
\end{equation}
where $|O(1)|\leq1$, and $g(z,0)\geq M-1$. This implies that
\begin{equation}\label{model6sup}
    \sup_{z}g(z,t)=\sup_{z} g(z,0)-\frac{t}{2\pi} +O(1),
\end{equation}
since $\sign g(z,0)=+1$ if $g(z,0)\geq M-1$. Since $\rho(t)=2\pi\sup_{z}g(z,t)$ we obtain the approximate equality
 \begin{equation}\label{approx_T3}
     {(\rho(0)-\rho(t))}/{t}=1+O(1/t),
 \end{equation}
provided that the length $T$ of the time interval is less than ${2\pi(M-1)}$.

By partition of any sufficiently long interval of time $[0,T]$ into many equal intervals of length $\leq{2\pi(M-1)}$ we
come to the following precise result.
\begin{theorem}\label{AsymptoticOptimality}
Consider evolution $\rho(t)=\rho(g_t)$ of $\rho$ under control \myref{model5}. Let
\begin{equation}
    M=\min\{\rho(0),\rho(T)\}.
\end{equation}
Suppose that $M\to+\infty$, $T\to+\infty$. Then, we have
\begin{equation}\label{approx_T}
    {(\rho(0)-\rho(T))}/{T}=1+O(1/{T}+1/{M}).
\end{equation}
Under any other admissible control,
\begin{equation}\label{approx_T2}
    {(\rho(0)-\rho(T))}/{T}\leq 1+O(1/{T}+1/{M}).
\end{equation}
\end{theorem}
The preceding arguments of this Section prove statement~(\ref{approx_T}) and statement~(\ref{approx_T2}) following from
Theorem~\ref{RhoEstimate}.

\section{Features of the dry-friction flow}\label{sec:features}

Methods used in Section~\ref{sec:existence} allows revealing the basic properties of the dry-friction flow. In
particular, it is possible to derive the asymptotic optimality of the dry-friction flow directly from
Eqs.~(\ref{model502})--(\ref{phi3}).

\subsection{Far from the target}
Indeed, it follows from equations \myref{model502}, \myref{model502+} that if $\sup_x g(x,0)>1/2$, then
$\sup_{t\in[0,2\pi]} \phi(t)=\sup_x g(x,0)-1/2$. The same estimates hold with essential supremum ${\rm vraisup}$ instead
of the plain $\sup$. In particular,
\begin{equation}\label{sup}
    \Vert\phi_0\Vert=\Vert g\Vert-\frac12,
\end{equation}
where $\phi_0$ is the restriction of $\phi$ to the interval $[0,2\pi]$, $g(x)=g(x,0)$ is the initial data, and $\Vert
f\Vert$ stands for the $L_\infty$-norm of $f$ over the same interval  $[0,2\pi]$. From equation \myref{phi3} it follows
that
\begin{equation}\label{sup2}
    \Vert\phi_1\Vert=\Vert \phi_0\Vert-1,
\end{equation}
where $\phi_1$ is the restriction of $\phi(t+2\pi)$ to the interval $[0,2\pi]$. We denote by $F_{\tau}$, where
$\tau\geq0$, shift of the argument $(F_{\tau}\phi)(t):=\phi(t+\tau)$. It is clear from the definition of the dry-friction
flow $\Phi_\tau$ and Eq.~(\ref{sup}) that
\begin{equation}\label{sup_tau}
    \Vert F_{\tau}\phi\Vert=\Vert \Phi_\tau g\Vert-\frac12,
\end{equation}
provided that $\Vert \Phi_\tau g\Vert\geq\frac12$. Eq.~(\ref{sup2}) means that
\begin{equation}\label{sup2tau}
    \Vert F_{2\pi}\phi\Vert=\Vert \phi\Vert-1.
\end{equation}
These identities imply that for natural integers $k$
\begin{equation}\label{sup_g}
    \Vert  \Phi_{2k\pi} g\Vert=\Vert  g\Vert-k,
\end{equation}
provided that $\Vert  g\Vert\geq k+\frac12$. This can be restated in the notations of the preceding
Section~\ref{sec:optimality} as follows:
\begin{equation}\label{sup_asymp}
    \frac{\rho(0)-\rho(2k\pi)}{2k\pi}=1,
\end{equation}
and gives a very precise form of the asymptotic optimality (see Theorem~\ref{AsymptoticOptimality}).

\subsection{Near the target}

On the contrary, if $\Vert  g\Vert\leq\frac12$, the dry-friction flow does not help to damp the string. Under this
condition, we obtain from Eq.~(\ref{model502}) and Eq.~(\ref{model502+}) that
\begin{equation}\label{attractg}
    \phi_0=0,\quad \sign\phi_0=-2g,
\end{equation}
and from equation \myref{phi3} we obtain that
\begin{equation}\label{attractg2}
    \phi(t+2\pi)=0,\quad \sign\phi(t+2\pi)=2g(t)=2g(t,0)=-\sign\phi(t),
\end{equation}
provided that $t\in[0,2\pi)$. This means that the dry-friction flow is given by solution of the Cauchy problem for the
linear equation
\begin{equation}\label{attractg3}
    \left(\frac{\partial }{\partial t}-\frac{\partial}{\partial x}\right)g(x,t)=-2(-1)^k\delta(x)g(t,0),\mbox{ if }t\in[2k\pi,2(k+1)\pi).
\end{equation}
The norm $\Vert  \Phi_{2k\pi} g\Vert=\Vert  \phi(t+2k\pi)\Vert$ does not depend on the natural integer $k$.

It is easy to solve the Cauchy problem~(\ref{attractg3}) explicitly. The solution $g(x,t)$ is determined via the initial
data $G(x)=g(x,0)$ as $g(x,t)=(-1)^k G(x+t)$ provided that $x\in [0,2\pi)$ and $t\in[2k\pi,2(k+1)\pi)$.

\section{Conclusion}\label{sec:conclusion}

The subject of the present paper have arisen as a natural extension of our preceding study of finite systems of
oscillators~\cite{Ovseevich2013,Ovseevich2016}. There we put an emphasis on the case of non-resonant systems. Here we
study the string which is an infinite system of highly resonant oscillators. In both cases the basic new results are the
existence and uniqueness of the motion under the dry-friction control and the asymptotic optimality of the control.

The results of both studies are similar, but methods are rather different. The similarity of studies is especially far
reaching in first parts of both of them, where we investigate the reachable sets and limiting capabilities of admissible
controls. In some aspects the present case of a string is simpler than that of finitely many non-resonant oscillators.
E.g., we do not use any special function, like the hypergeometric function in the sense of Gelfand or DiPerna--Lions
theory~\cite{DiPernaLions1989}, which play a decisive part in control of finite system of oscillators (see
Refs.~\cite{Ovseevich2016}). On the other hand the infinite dimensional case is related to well known and quite real
analytic difficulties which are present (for details, see Appendix~III).

\bigskip
\noindent{\bf Acknowledgements}. This work was supported by the Russian Scientific Foundation, grant 16-11-10343.

{\setcounter{equation}{0} \setcounter{theorem}{0}\setcounter{lemma}{0} \setcounter{corollary}{0}
\def\theequation{A\arabic{equation}}\def\thetheorem{A\arabic{theorem}} \def\thelemma{A\arabic{lemma}}
\def\thecorollary{A\arabic{corollary}}
\newcounter{append}
\setcounter{append}{0}
\def\theappend{\Roman{append}}\refstepcounter{append}

\section*{APPENDIX \theappend. Singular arcs I}\refstepcounter{append}\label{sec:AppI}

The above analysis of the dry-friction control near the target extends to more general analysis of the motion along
singular arcs. These are by definition the time-intervals, where in the controlled motion
\begin{equation}\label{model_sing}
    \left(\frac{\partial }{\partial t}-\frac{\partial}{\partial x}\right)g(x,t)=\delta(x)u(t),\quad u=-\sign g(0,t)
\end{equation}
the control is not uniquely defined by the current state of the string, i.e. $g(0,t)\equiv0$.

In order to construct a motion of this kind we use the spectral decompositions $g(x,t)=\sum g_\mu(x)e^{i\mu t}$, and
$u(t)=\sum u_\mu e^{i\mu t}$. This almost periodic function should be bounded: $|u|\leq1$ Then, the functions $g_\mu$
should satisfy
\begin{equation}\label{gmu}
    i\mu g_\mu-\frac{\partial}{\partial x}\,g_\mu=-\delta u_\mu, \mbox{ and } g_\mu(0)=0.
\end{equation}
The first equation~(\ref{gmu}) guarantees that $0$ is the point of discontinuity of $g_\mu$, so that the second
equation~(\ref{gmu}) should be treated cautiously. In fact, the discussion of Section~\ref{sec:existence} shows that we
have to take $\frac12(g_\mu(0+)+g_\mu(0-))$ for $g_\mu(0)$. Indeed, according to the first  equation~(\ref{gmu}), the
function $g_\mu$ is piecewise differentiable with jumps at $x=0$. Therefore $2\pi$-periodic function $g_\mu$ should have
the form
\begin{equation}\label{gmu2}
    g_\mu(x)=C_\mu e^{i\mu x} \mbox{ for }x\in[0,2\pi),
\end{equation}
where the constant $C_\mu=(1-e^{2\pi i\mu})^{-1}u_\mu$. The condition $g_\mu(0)=0$ gives
\begin{equation}
    1+e^{2\pi i\mu}=0,
\end{equation}
which implies that $\mu$ should have the form $\mu=\frac12\nu$, where $\nu$ is an odd integer, and $C_\mu=u_\mu/2$.
Therefore, the control $u(t)=\sum u_\mu e^{i\mu t}$ is not just almost periodic but $4\pi$-periodic. Moreover, we have
\begin{equation}
    g(x,t)=\sum g_\mu(x)e^{i\mu t}=\frac12\sum u_\mu e^{i\mu (t+x)}=\frac12 u(t+x) \mbox{ for }x\in[0,2\pi).
\end{equation}

\section*{APPENDIX \theappend. Singular arcs II}\refstepcounter{append}\label{sec:AppII}

It is possible in a more general fashion analyze the singular arcs of the motion governed by the second-order nonlinear
wave equation~(\ref{string_eq220}). We will do this by the direct finite-dimensional approximation of the string by
finite number of harmonics. These arcs are the time-intervals of the motion, where the semi-flow
${\F}\to\phi_t({\F})={\F}_{t}$ leaves invariant the ``hyperplane'' $f_1(0)=0$, where the control $u=-\sign(f_1(0))$ is
not uniquely defined. We put the word hyperplane into quotation marks because the value $f_1(0)$ is badly defined within
the natural state space of the string. In order to be correct, we introduce a cut-off. Namely, we start with an
approximation of the string with first $N$ harmonics:
\begin{equation}\label{string_finite0}
\begin{array}{l}
    f_0(x)=\sum_{k=0}^N a_k\cos kx \mbox{ modulo constants }a_0,\\[.5em]
    f_1(x)=\sum_{k=0}^N b_k\cos kx,\\
\end{array}
\end{equation}
and consider in this $2N+1$-dimensional space the correctly defined ODE
\begin{equation}\label{string_finite}
\begin{array}{l}
    \dot a_0=b_0\\
    \dot b_0=u,\\
    \end{array}\quad\begin{array}{l}
    \dot a_k=b_k\\
    \dot b_k=-k^2a_k+2u,\quad k=1,\dots,N,\\
\end{array}
\end{equation}
which is the natural finite-dimensional approximation of the wave equation.

We require that $\sum_{k=0}^N b_k=0$ which is a restatement of the condition $f_1(0)=0$. This immediately implies that
$u=\frac1{2N+1}\sum_{k=0}^N k^2a_k$. Thus, we obtain a linear differential equation
\begin{equation}\label{string_finite2}
\begin{array}{l}
    \dot a_k=b_k\\
    \dot b_k=-k^2a_k+\frac1{2N+1}\sum_{k=1}^N k^2a_k,\quad k=1,\dots,N.\\
\end{array}
\end{equation}
in the space $\R^{2N}$ of sequences $a_k,b_k,\, k=1,\dots,N$. We are going to solve system~(\ref{string_finite2}) and
then remove the cut-off: In other words, we pass to the limit $N\to\infty$ in the obtained solution. To obtain the
solution we first find the ``spectral decomposition'' of the linear operator given by the RHS of
Eq.~(\ref{string_finite2}). The corresponding eigenvalue problem is:
\begin{equation}\label{string_finite3}
\begin{array}{l}
    \lambda a_l=b_l\\
    \lambda b_l=-l^2a_l+\frac2{2N+1}\sum_{k=1}^N k^2a_k,\quad l=1,\dots,N,\\
\end{array}
\end{equation}
which is equivalent to
\begin{equation}\label{string_finite4}
    (l^2+\lambda^2) a_l=\frac2{2N+1}\sum_{k=1}^N k^2a_k,\quad l=1,\dots,N.\\
\end{equation}
This, in turn, is equivalent to the system
\begin{equation}\label{string_finite40}
\begin{array}{l}
    a_k=\frac{R}{k^2+\lambda^2},\\[1em]
    \displaystyle\frac2{2N+1}\sum_{k=1}^N
    \frac{k^2}{k^2+\lambda^2}=1.
\end{array}
\end{equation}
Here $R=R_\lambda$ is an arbitrary $k$-independent constant. Thus, the eigenvalue problem reduces to the solution of
\begin{equation}\label{string_finite5}
    \frac2{2N+1}\sum_{k=1}^N \frac{k^2}{k^2+\lambda^2}=1.\\
\end{equation}
It easy to show that the problem (\ref{string_finite5}) has $2N$ purely imaginary eigenvalues $\lambda=i\mu$. Indeed,
this statement is equivalent to the fact that the polynomial equation of degree $N$
\begin{equation}\label{string_finite6}
    \frac2{2N+1}\sum_{k=1}^N \frac{k^2}{k^2-t}=1.\\
\end{equation}
has $N$ real roots. It is clear, at least when $N$ is large, that there is a positive root $t$ in a close vicinity of
zero, and that the function $t\mapsto\sum_{k=1}^N \frac{k^2}{k^2-t}$ tends to $+\infty$ as $t\to k^2-0$, and tends
$-\infty$ as $t\to k^2+0$ for $k=1,\dots,N$. This implies that each interval $[k,k+1]$ for $k=0,\dots,N-1$ contains a
root $t_k$.

Now we pass to the limit $N\to\infty$ in Eq.~(\ref{string_finite5}) and Eq.~(\ref{string_finite6}). To do this we rewrite
Eq.~(\ref{string_finite5}) in the following form:
\begin{equation}\label{string_finite7}
    \sum_{k=1}^N \frac{k^2-\mu^2}{k^2-\mu^2}+\mu^2\sum_{k=1}^N \frac{1}{k^2-\mu^2}=N+\frac12,\\
\end{equation}
which is equivalent to
\begin{equation}\label{string_finite8}
    \sum_{k=1}^N \frac{1}{k^2-\mu^2}=\frac1{2\mu^2}.\\
\end{equation}
The function $g_N(\mu)=\sum_{k=1}^N \frac{1}{k^2-\mu^2}$ has then a well-defined limit as $N\to\infty$. Namely,
\begin{equation}\label{string_finite9}
    \lim _{N\to\infty}g_N(\mu)=\sum_{k=1}^\infty \frac{1}{k^2-\mu^2}-\frac{1}{2\mu^2}=\frac{\pi}{2\mu}\ctg(\pi\mu)=g(\mu),
\end{equation}
and the limit eigenvalues $i\mu$ are given by roots $\mu_k=k+\frac12$ of $\ctg(\pi\mu)$. Here $\mu_k$ is the unique
solution of $\ctg(\pi\mu)$ in the interval $[k,k+1]$. The spectral decomposition is defined by the correspondence
\begin{equation}\label{string_finite10}
    a_k=\sum_{\mu}\frac{R_\mu}{k^2-\mu^2}
\end{equation}
between infinite sequences $a_k$ and $R_\mu$.

\medskip
\noindent{\bf Remark.} We have the following result:
\begin{theorem}\label{function} (cf. the ``Eisenstein''-part of Ref.~\cite{Weil1976})
    \begin{equation}\label{function0}
    \sum_{k=1}^\infty \frac{1}{k^2-\mu^2}\cos kx-\frac1{2\mu^2}=-\frac1{|\mu|}\sin \Vert \mu x\Vert_\mu,
    \mbox{ \rm where }
    \Vert y\Vert_\mu=\inf_{n\in\Z}|y+2\pi\!{\mu}{n}|
    \end{equation}
\end{theorem}

\begin{proof}
Indeed, operator $L=\frac{\partial^2}{\partial x^2}+\mu^2$, when applied to LHS and RHS of (\ref{function0}) gives
$-\delta(x)$, and the kernel of $L$ in the space of $2\pi$-periodic functions is 0.
\end{proof}

This fact allows to rewrite the transform  \myref{string_finite10} in its functional form as follows:
\begin{equation}\label{string_finite100}
\begin{array}{l}
f_0(x)=\sum a_k\cos kx=-\sum_{\mu}\frac{\Re{R_\mu}}{|\mu|}\sin \Vert \mu x\Vert_\mu\\[1em]
f_1(x)=\sum b_k\cos kx=-\sum_{\mu}(\sign\mu)\Im{R_\mu}\sin \Vert \mu x\Vert_\mu
\end{array}
\end{equation}

The  sequence of complex numbers $R_\mu$ is self-adjoint, meaning that $R_{-\mu}=\overline{R}_{\mu}$. Thus, the solution
of Eq.~(\ref{string_eq22}) along a singular arc has the form $f(x,t)=\sum a_k(t)\cos kx$, where
\begin{equation}\label{string_finite1000}
    a_k(t)=\sum_{\mu}\frac{R_\mu e^{i\mu t}}{k^2-\mu^2}=2\sum_{\mu>0}\frac{\Re (R_\mu e^{i\mu t})}{k^2-\mu^2},
\end{equation}
and $\mu$ runs over roots $\mu_k=k+\frac12$ of the function $\tg(\pi\mu)$.

The point with coordinates $a_k, b_k$ belongs to the singular arc at the cut-off level $N$ if the control
$u=\frac1{N+1}\sum_{k=0}^N k^2a_k$ satisfies the bound $|u|\leq1$. The problem is to pass to the limit $N\to\infty$ in
this condition. In terms of the variables $R_\mu$ the condition tells that
\begin{equation}\label{string_finite12}
    \left|\frac1{N+1}\sum_{k\in[1,N],\mu\in (0,N)} \frac{k^2 \Re (R_{\mu})}{k^2-\mu^2}\right|\leq\frac12.
\end{equation}
The  sum
\begin{equation}
    \frac1{N+1}\sum_{k=1}^{k=N}\frac{k^2}{k^2-\mu^2}=\frac1{N+1}(N+\mu^2 g_N(\mu))=\frac1{N+1}(N+1+\mu^2 g_N(\mu)-1)
\end{equation}
equals to $1$ because of equation (\ref{string_finite8}): $\mu^2 g_N(\mu)=1$. Therefore, the condition
(\ref{string_finite12}) at the cut-off level $N$ can be restated as follows:
\begin{equation}\label{string_finite14}
    \left|\sum_{\mu\in (0,N)}  \Re (R_{\mu})\right|\leq\frac12.
\end{equation}
The formal passage to the limit $N\to\infty$ transforms  (\ref{string_finite14}) into
\begin{equation}\label{string_finite13}
    \left|\sum_{\mu>0}  \Re (R_{\mu})\right|\leq\frac12.
\end{equation}
We note that the summation in Eq.~(\ref{string_finite14}) and Eq.~(\ref{string_finite13}) goes over different sets of
roots $\mu$: in the first case over roots of $\mu^2 g_N(\mu)=1$, and in the second case over roots of $\mu^2 g(\mu)=1$.

The time-limits of the arc are determined by the inequality
\begin{equation}\label{string_finite11}
    \left|\sum_{\mu>0} \Re (R_{\mu} e^{i\mu t})\right|\leq\frac12,
\end{equation}
which says that the control $u$ satisfies $|u|\leq1$. Eq.~(\ref{string_finite12}) and Eq.~(\ref{string_finite1000})
together determine at a formal level the singular motion of the system. In order to make these considerations applicable
to the ``real'' string, we need to know the continuity properties of the function
\begin{equation}
    t\mapsto\sum_{\mu>0} \Re (R_{\mu}e^{i\mu t}),
\end{equation}
when the function $\frac{\partial f_0}{\partial x}+f_1$ is bounded, and related to the sequence $R_{\mu}$ via
Eq.~(\ref{string_finite1000}).

\section*{APPENDIX \theappend. Contracting properties of the dry-friction control}\refstepcounter{append}\label{sec:AppIII}

We consider the semi-flow ${\F}\to\phi_t({\F})={\F}_{t}$, defined by Cauchy problem for the nonlinear string
equation~(\ref{string_eq2}) of second order. Here we will show that at formal level the semi-flow
${\F}\to\phi_t({\F})={\F}_{t}$ is continuous, and even contracting wrt the ``spacial'' argument ${\F}$. Indeed, consider
the functional of energy
\begin{equation}
    E({\F})=\frac12 \left\Vert f_1\right\Vert^2+\frac12\left\Vert \frac{\partial f_0}{\partial x}\right\Vert^2,
\end{equation}
where $\Vert  \cdot\Vert$ is the $L_2(0,2\pi)$-norm. We have the following basic a priori estimate
\begin{equation}\label{string_estimate}
    \frac{d}{dt} E(\phi_t({\F})-\phi_t({\mathfrak g}))\leq0.
\end{equation}
Indeed, by formal computation we have for $E(t)= E(\phi_t({\F})-\phi_t({\mathfrak g}))$ that
\begin{equation}\label{string_estimate2}
    \begin{array}{l}
    \frac{d}{dt} E=A+B+C, \mbox{ where }\\[1em]
    A=\left\langle \frac{\partial }{\partial x}\left(f_{t0}-g_{t0}\right),\frac{\partial}{\partial x}\left(f_{t1}-g_{t1}\right)\right\rangle,  \\[1em]
    B=\left\langle f_{t1}-g_{t1}, \Delta\left(f_{t0}-g_{t0}\right)\right\rangle,\\[1em]
    C=-\left(f_{t1}(0)-g_{t1}(0)\right)\left(\sign f_{t1}(0)-\sign g_{t1}(0)\right).\\
\end{array}
\end{equation}
We have $A+B=0$, because 
\begin{equation}
	\langle v, \Delta u\rangle=-\langle \frac{\partial v}{\partial x}, \frac{\partial u}{\partial x}\rangle
\end{equation}
for any pair $u,v$ of periodic functions. Therefore,
\begin{equation}
    \dot E=A+B+C=C.
\end{equation}
Denote $x=f_{t1}(0)$, $y=g_{t1}(0)$. 
We then have 
\begin{equation}
	C=-(x-y)(\sign x-\sign y)\leq0, 
\end{equation}
because the sign-function is monotone.
This implies inequality~(\ref{string_estimate}) which, in turn, implies that the map ${\F}\mapsto\phi_t({\F})$ is
contracting wrt the energy-norm for $t\geq0$. These considerations are formal, because are based on the formal
differentiation of a product of distributions.

This becomes even more clear if we rewrite the above formal computation in the simpler case of the nonlinear string
equation~(\ref{model}) of the first order. In this case the energy is as follows:
\begin{equation}
    E(g)=\frac12 \left\Vert g\right\Vert^2=\frac12\int  g(x)^2\,dx,
\end{equation}
where integration is over the torus $\R/2\pi\Z$. Formally, if 
\begin{equation}
	E(t):=E(\Phi_t(g)-\Phi_t({f})), 
\end{equation}
then
\begin{equation}\label{string_estimate20}
    \begin{array}{l}
    \frac{d}{dt} E=A+B, \mbox{ where }\\[1em]
    A=\left\langle \frac{\partial }{\partial x}\left(f-g\right),\left(f-g\right)\right\rangle,  \\[1em]
    B=-\left(f(0)-g(0)\right)\left(\sign f(0)-\sign g(0)\right).\\
\end{array}
\end{equation}
The term $A=0$ for any pair of periodic functions $f,\,g$, and the term $B$ is nonpositive since  the sign-function is
monotone. Thus, the dry-friction semi-flow $\Phi_t$ is contracting wrt the energy-norm.

It is not clear to us whether the dry-friction flow rigorously constructed in Section~\ref{sec:existence} is contracting
indeed wrt the energy norm.

\end{document}